\documentclass[12pt]{article}
\usepackage[utf8]{inputenc}
\usepackage[centertags]{amsmath}
\usepackage{hyperref}
\usepackage{amsfonts}
\usepackage{amssymb}
\usepackage{latexsym}
\usepackage{amsthm}
\usepackage{newlfont}
\usepackage{graphicx}
\usepackage{listings}
\usepackage{booktabs}
\usepackage{abstract}
\usepackage{enumerate}
\usepackage{xcolor}
\usepackage[normalem]{ulem}
\usepackage{authblk} 
\RequirePackage{srcltx}
\date{}

\newlength{\defbaselineskip}
\newcommand{\setlinespacing}[1]%
           {\setlength{\baselineskip}{#1 \defbaselineskip}}

\newcommand{\actaqed}{\hfill $\actabox$}
{\medskip\noindent \textit{Proof of #1. }}%
{\actaqed \medskip}

\def\cC{{\mathcal C}}
\def\cD{{\mathcal D}}
\def\D{{\mathcal D}}

\def\cR{{\mathcal R}}
\def\cS{{\mathcal S}}

\def\cX{{\mathcal X}}

\def\bbC{{\mathbb C}}

\def\bbN{{\mathbb N}}
\def\N{{\mathbb N}}

\def\R{{\mathbb R}}

\def\bbT{{\mathbb T}}

\def\bx{\mathbf x}

 \def \<{\langle}
\def\>{\rangle}

\def \Og{\Omega}

\def \e{\varepsilon}

\def \de{\delta}

\def \ff{\varphi}

\def\vi{\varphi}

\def \sp{\operatorname{span}}

\newtheorem{Theorem}{Theorem}[section]
\newtheorem{Lemma}{Lemma}[section]
\newtheorem{Definition}{Definition}[section]
\newtheorem{Proposition}{Proposition}[section]
\newtheorem{Remark}{Remark}[section]

\newtheorem{Corollary}{Corollary}[section]
\numberwithin{equation}{section}

\newtheorem{OldTheorem}{Theorem}[section]

\newcommand{\be}{\begin{equation}}
\newcommand{\ee}{\end{equation}}

\title{Some theoretical and practical results on noisy signals recovery}

\author[1]{M. Makurin\thanks{makurin@gmail.com}}
\author[2,3,4]{Yu. Malykhin\thanks{malykhin-yuri@yandex.ru}}
\author[2,4,5]{K. Ryutin\thanks{kriutin@yahoo.com}}
\author[2,3,4,5]{V.~Temlyakov\thanks{temlyakovv@gmail.com}}

\affil[1]{Moscow RTT Laboratory, Moscow, Russia}
\affil[2]{Marchuk Institute of Numerical Mathematics of the Russian Academy of Sciences, Moscow, Russia}
\affil[3]{Steklov Mathematical Institute of Russian Academy of
Sciences, Russia}
\affil[4]{Lomonosov Moscow State University, Russia}
\affil[5]{Moscow Center of Fundamental and Applied Mathematics, Russia}

\begin{document}

\maketitle

\begin{abstract}{In this paper we discuss some theoretical results and their applications to specific practical problems from 
wireless communications. We consider both the Hilbert and the Banach  space settings. We assume that we deal with a noisy  version of a signal (clean signal), which is  sparse  with respect to a given system of elements (dictionary). We assume that we know its noisy version  at a finite number of points and we want to approximately recover it. This  problem of recovery of a noisy signal is closely related to 
the problem of establishing the Lebesgue-type inequalities for the corresponding algorithms and it motivates us to prove such inequalities.  Under certain conditions on a dictionary, which are popular in practical applications (RIP-type condition, coherence condition) we obtain different kinds of the Lebesgue-type inequalities for the OMP and its version WOMP algorithms. The most important feature of our new theoretical result is the assumption  that the dictionary has the RIP-type property instead of the assumption that it is the Riesz basis, which was used in the previous results. Our approach  allows us to treat redundant (overcomplete) systems, which is important in applications.

We also consider the setting of sparse recovery
for highly-coherent dictionaries that appear in  OFDM   setting   for  wireless
communication. Our main example is the oversampled
Fourier dictionary and  the recovery of the frequency response of a sparse channel. We develop a  general approach to such problems and we prove that under certain conditions the OMP
algorithm recovers all dictionary elements with
large coefficients, and estimate its accuracy  (in NMSE metric) in terms of signal-to-noise ratio.
We also impose extra (in addition to being sparse) assumptions on the clean signal. In particular, these extra assumptions are expressed in the form of fast decay of coefficients of an expansion or in the form of an assumption that the clean signal is a linear combination of elements with much better than for the whole dictionary coherence properties. We show that under these assumptions the OMP and some other greedy-type algorithms work well. For completeness we discuss some known 
results on  the Lebesgue-type inequalities in Banach spaces.

}
\end{abstract}

\section{Introduction}
\label{In}

In this paper we discuss some theoretical results and their applications to specific practical problems from the
wireless communications. Mostly, we discuss the Hilbert space setting  and present some theoretical results on the Banach space setting in Section \ref{Disc}. We begin with a description of  the problem of approximate recovery of noisy signals.   We present it in the Banach space setting. Then we formulate some known results from the greedy approximation and compressed sensing theories and explain their connections to the noisy signals recovery problem.  After that we discuss some new results on the noisy signals recovery, which address specific problems from the wireless communications. The reader can learn about history or other  aspects of these theories from \cite{VT211},\cite{FR}.

{\bf Motivation. The channel estimation problem.}
Our paper grew out  of  our studies of  channel  estimation problem from the wireless communication theory.  We give some  details of this problem in Subsection \ref{sub_channel}.

{\bf General setting. Assumptions on the noise.} Let $X$ be a separable Banach space with the norm $\|\cdot\|_X$.  Usually, in the greedy approximation and compressed sensing literature the noisy data is understood in the deterministic sense. Namely, we assume that our original (clean) signal $f^*$ has some good properties  but we measure the signal $f$, which is a noisy version of $f^*$:
\be\label{In1}
\|f-f^*\|_X \le \epsilon. 
\ee
Normally, we want to recover our clean signal with small error measured in the norm of the space $X$. Clearly, if the theory guarantees a bound for $\|f-a(f)\|_X$ for an approximant $a(f)$ then for the clean signal recovery error we have
\be\label{In2}
\|f^* -a(f)\|_X \le \|f -a(f)\|_X + \|f-f^*\|_X \le \|f -a(f)\|_X + \epsilon. 
\ee
Along with the standard assumption (\ref{In1}) on the noise there are other versions of that assumption, which are used in the literature. 

In the greedy approximation and compressed sensing theories we use  a system $\cD$ of elements of $X$ for approximating a given signal $f\in X$. This system is called a dictionary and we assume that it has two properties: (I) $\cD=\{g\}$, $\|g\|_X\le1$ for all $g\in \cD$, (II) the closure (in $X$) of the span of $\cD$ coincides with $X$.  In the case of a Hilbert space $\mathcal{H}$ we introduce a new norm, associated with a dictionary $\cD$,  by the formula
\begin{equation}
    \label{maxcorr_norm}
\|f\|_\cD:=\sup_{g\in\cD} |\<f,g\>|,\quad f\in \mathcal{H}.
\end{equation}
It turns out that the norm $\|\cdot\|_\cD$ is useful in measuring the noise by imposing the assumption
\be\label{In5h}
\|f-f^*\|_\cD \le \epsilon. 
\ee

{\bf General setting. Assumptions on the clean signal.} We list some typical assumptions.

{\bf A1. Strict sparsity.} Usually, {\it sparse representation} of a signal (function) $h$ means that there is a given system of functions (dictionary) $\cD$ such that $h$ has a  representation as a linear combination of no greater than (or exactly) $K$ elements from the dictionary ($h$ is $K$-sparse with respect to $\cD$):
$$
h=\sum_{j=1}^K c_jg_j,\quad g_j\in\cD,\quad j=1,\dots,K.
$$
We denote the set of all such vectors $h$ as $\Sigma_K(\cD).$

 It is clear that in the case of a finite dictionary the assumption {\bf A1} is equivalent to the assumption that $\sigma_K(h,\cD)=0$, where 
 $$
\sigma_K(h,\cD)_X := \inf_{g\in \Sigma_K(\cD)} \|h-g\|_X
$$
is the best $K$-term approximation of $h$. 
A weaker form of the sparsity assumption {\bf A1} is the following.

{\bf A2. Approximate sparsity.} We describe sparsity properties of $h$ by the rate of decay of the sequence $\{\sigma_m(h,\cD)\}_{m=1}^\infty$ of best $m$-term approximations of $h$ with respect to $\cD$. In other words, this assumption means that $h$ belongs to some approximation class. 

The following geometrical assumption turns out to be very useful in studying nonlinear sparse approximations.

{\bf A3. Compressibility.} Assume that $h$ has a representation
$$
h=\sum_{j=1}^\infty c_jg_j,\quad g_j\in\cD,\quad j=1,\dots,\quad \sum_{j=1}^\infty|c_j| \le R.
$$
In other words this means that $h$ belongs to the generalised octahedron of the symmetrised version of the dictionary $\cD$. 
It is convenient for us to consider along with the dictionary $\cD$ its symmetrization. In the case of real Banach spaces we 
denote 
$$
\cD^{\pm} := \{\pm g \, : \, g\in \cD\}.
$$
In the case of complex Banach spaces we denote 
$$
\cD^{\circ} := \{e^{i\theta} g \, : \, g\in \cD,\quad \theta \in [0,2\pi)\}.
$$
In the above notation $\cD^{\circ}$ symbol $\circ$ stands for the unit circle. 

{\bf A4. Fast decay of coefficients.} For a given $q\in (0,1)$ assume that $h\in X$  has an expansion  
$$
h= \sum_{k=0}^\infty a_k g_k,\quad g_k\in \cD,\quad |a_k| \le q^k,\quad k=0,1,\dots.
$$

{\bf General setting. Assumptions on the dictionary. Hilbert space.} In the majority of cases we need to impose restrictions on the dictionary in order to be able to prove theoretical results. We present some standard assumptions on the dictionary.  Let $\cD$ be a dictionary in a Hilbert space $\mathcal{H}$. We define the coherence parameter of this dictionary in the following way
$$
M(\cD):=M(\cD,\mathcal{H}):= \sup_{g\neq h;g,h\in\cD} |\<h,g\>|.
$$
Very often it is difficult to calculate the coherence parameter $M(\cD)$ exactly. For this reason, it is convenient to use 
some majorant $\mu$, i.e. a number such that $M(\cD) \le \mu$. 

We now proceed to the RIP (Restricted Isometry Property) dictionaries, which are popular in compressed sensing. 
\begin{Definition}\label{InD1} A dictionary $\cD:=\{g\}$ is called the Riesz dictionary with depth $u$ and parameters $0<R_1\le R_2 <\infty$ in a Hilbert space $\mathcal{H}$,  if, for any $u$ distinct elements $g_1,\dots,g_u$ of the dictionary and any coefficients $a=(a_1,\dots,a_u)$, we have
\begin{equation}\label{In6}
R_1\|a\|_2 \le \left\|\sum_{i=1}^u a_ig_i\right\|_\mathcal{H}\le R_2\|a\|_2.
\end{equation}
We denote the class of Riesz dictionaries with depth $u$ and parameters $0<R_1\le R_2 <\infty$  by $\mathcal{R}(u,R_1,R_2)$.  If $\cD$ is a basis and (\ref{In6}) holds for all $u$, then $\cD$ is called the Riesz basis. 
\end{Definition}
In the case  $\delta \in (0,1)$ and $R^2_1=1-\de$, $R^2_2=1+\de$, we write $\mathcal{R}(u,\de)$ instead of $\mathcal{R}(u,R_1,R_2)$ and call a dictionary $\cD \in \mathcal{R}(u,\de)$ the Riesz dictionary with depth $u$ and parameter $\de$.
The term Riesz dictionary with depth $u$ and parameter $\delta \in (0,1)$ is another name for a dictionary satisfying the Restricted Isometry Property (RIP) with parameters $u$ and $\de$. 
 
  Let  $\Omega\subset \R^d$ be a nonempty set (usually a compact set)  equipped  with a  probability measure $\mu$.  Following traditions of mathematical analysis, we use the standard notation
$\mu$ for the measure on $\Omega$. The same notation $\mu$ (standard
in signal processing)
 is also used for the majorant of the coherency $M(D)$ (see above). However, we
are sure that this coincidence
will not cause any confusion for the reader.
For  $1\le p\leq  \infty$, we let  $L_p(\Omega):=L_p(\Omega,\mu)$  denote  the  Lebesgue  space $L_p$ defined with respect to the measure $\mu$ on $\Omega$ and let 
$$
\|f\|_p := \|f\|_{L_p(\Omega)} := \|f\|_{L_p(\Omega,\mu)} := \left(\int_\Og |f|^pd\mu\right)^{1/p}
$$  
be the  norm of $L_p(\Og)$ for $1\le p<\infty$. We will use a slight abuse of the notation that $L_\infty(\Og)$ denotes the space of all uniformly bounded measurable functions on $\Og$ with the norm
$\|f\|_\infty=\sup_{x\in \Og} |f(x)|$.  As usual $\delta_\bx$ denotes the Dirac measure supported at a point $\bx$.
 
 Throughout the paper we denote by  
$ \Phi_N:=\{\ff_j\}_{j=1}^N$  a dictionary of $N$  uniformly bounded functions on $\Og \subset \R^d$ satisfying 
\be \label{I9}\sup_{\bx\in\Og} |\vi_j(\bx)|\leq 1,\   \ 1\leq j\leq N.\ee

  The following Theorem \ref{IT3} is the main general theoretical result of this paper.  See Section \ref{O} for explanation of the notations and the proof. 
 
\begin{Theorem}\label{IT3}  Let a natural number $s\le N$ be given.   Assume
    that $\Phi_N$ belongs to $\cR(s,R_1,R_2)$ for some constants $0<R_1\leq
    R_2<\infty$ and is a uniformly bounded system
    satisfying  (\ref{I9}).  For given parameters $t\in (0,1]$ and $R_1$, $R_2$
    and some  $c:= c(t,R_2/R_1)$ the following holds.

Let $\xi^1,\ldots, \xi^m$ be i.i.d.
random points  with common  distribution  $\mu$ on $\Omega$. Then there exists an absolute constant $C_0$ such that with high probability for 
\begin{equation}
    \label{IT3_m}
m \ge  C_0  Ks \log N\cdot (\log(2Ks ))^2\cdot (\log (2Ks )+\log\log N),\quad K:= R_1^{-2}, 
\end{equation}
    we have the following properties: For any integer $v$ such that
    $(1+c)v\le s$ and
    any $f_0\in \cC(\Omega)$ the WOMP  algorithm  with weakness parameter $t$ applied to $f_0$ with respect to the   $\Phi_N(\Omega_m)$ in the space $L_2(\Omega_m,\mu_m)$ provides
\be\label{Inmp}
\|f_{c v}\|_{L_2(\Omega_m,\mu_m)} \le C\sigma_v(f_0,\Phi_N(\Omega_m))_{L_2(\Omega_m,\mu_m)}, 
\ee
and
\be\label{Inmp2}
\|f_{c v}\|_{L_2(\Omega,\mu)} \le C'\sigma_v(f_0,\Phi_N)_\infty 
\ee
with absolute constants $C$ and $C'$.
Moreover, we have
 \be\label{Inmp3}
  \|f_{cv} \|_{L_2(\Omega,\mu)} \le C'' \sigma_v(f_0,\Phi_N)_{L_2(\Og, \mu_\xi)}
 \ee
 with an absolute constant $C''$ and 
 $$
\mu_\xi := \frac{1}{2} \mu + \frac{1}{2m} \sum_{j=1}^m \delta_{\xi^j}.
$$
 \end{Theorem}

Let us discuss Theorem \ref{IT3}. First of all, in the Lebesgue-type inequalities we have in the left hand sides of 
(\ref{Inmp}) -- (\ref{Inmp3}) the error of approximation of $f_0$ with respect to the system $\Phi_N$ in different 
norms: $\|\cdot\|_{L_2(\Omega_m,\mu_m)}$ in (\ref{Inmp}) and $\|\cdot\|_{L_2(\Omega,\mu)}$ in (\ref{Inmp2}), (\ref{Inmp3}). Also,
in the right hand sides of 
(\ref{Inmp}) -- (\ref{Inmp3}) the best $v$-term approximations of $f_0$ with respect to the system $\Phi_N$ are measured in different norms.
This means that it corresponds to the sparsity assumption {\bf A1} and the noise assumptions (\ref{In1}).
We guarantee those inequalities under serious assumptions on the system $\Phi_N$. Sometimes, the assumption that $\Phi_N$ is a uniformly bounded system from $\cR(s,R_1,R_2)$ is easy to check. For many classical systems it is known. For instance, in the case
$\Phi_N$ is the trigonometric system $\{e^{ikx}\}_{k=1}^N$ defined on $\bbT=[0,2\pi]$ this assumption is satisfied with $s=N$, $R_1=R_2=1$ -- the system is orthonormal. 
Theorem \ref{IT3} is a development of results from \cite{DTM2}. The most important new feature of Theorem \ref{IT3} is 
the assumption  that $\Phi_N$ belongs to $\mathcal{R}(s,R_1,R_2)$ instead of the assumption that $\Phi_N$ is the Riesz basis of the span of $\Phi_N$. It allows us to treat redundant systems, which is important in applications. 

We remark that in fact $c:=c(t,R_2/R_1)= c_{\ref{ssrT1}}(t,3^{1/2}R_2/R_1)$, where $c_{\ref{ssrT1}}(t,U)$ is from Theorem \ref{ssrT1}.

In Section \ref{FD} we mainly focus on  dictionaries with high coherence (close to $1$).
We discuss the following special setting. Let $\cD$ be a dictionary of a Hilbert space $\mathcal{H}$. 
Assume that this dictionary has a coherence parameter $M(\cD) \le \mu  <1$. In addition, consider a subdictionary 
$\cD_1 \subset \cD$ with a much smaller coherence parameter $M(\cD_1) \le \mu_1  <\mu$. Then, assuming that 
$$
f= \sum_{n=1}^L c_ng_n,\quad g_n \in \cD_1
$$ 
and  imposing  technical restrictions on the coefficients $\{c_n\}$ and the number $K\in \bbN$ we prove that the Pure Greedy Algorithm (PGA) (see definition below) chooses the elements $g_1,g_2,\dots,g_K$ at the first $K$ iterations.

In Section \ref{FD} we also establish a general result  that allows us  to deal with the channel estimation problem for wireless communication.

Let $\mathcal D$ be a finite dictionary.
We assume that the coherence $M(\mathcal D)<1$, but it may be close to $1$. So
we need a modification of the abovementioned results.

Our main assumption is that all elements of a clean signal lie in a sub-dictionary
$\cD_1\subset\cD$ that has small coherence. 
Let us introduce the following quantity:
$$
M^+(\cD_1,\cD) := \max_{g\in\mathcal D} \left\{\sum_{h\in\mathcal
D_1}|\langle g,h\rangle| - \max_{h\in\mathcal D_1}|\langle g,h\rangle|\right\}.
$$

We prove that if the parameter $M^+(\cD_1,\cD)$ is small, then the OMP algorithm
 recovers all elements with large coefficients, as in Theorem~\ref{th_tropp} from Section \ref{FD}.

\begin{Theorem}
    \label{thm_mu_plus}
    Let $\mathcal D_1\subset\mathcal D$ satisfy the conditions
    $M(\mathcal D)\le \mu$, $M^+(\mathcal D_1,\mathcal D)\le
    \mu^+$ for some positive numbers $\mu$, $\mu^+$ such that
    \begin{equation}
        \label{mu_plus_condition}
    \mu + 2\mu^+ < 1.
    \end{equation}
    Suppose that we observe a signal
    $$
    y = \sum_{k=1}^K c_k g_k + w,\quad g_k\in\mathcal D_1.
    $$
    Given $\theta>0$, we denote $\hat\beta := 2\theta/(1-\mu-2\mu^+)$ and assume
    that the noise $w$ satisfies the condition
    $$
        \|w\|_\cD \le \theta.
    $$
   
    Then the OMP algorithm with the dictionary $\mathcal D$, the initial vector
    $y$ and the stopping criteria
    $$
    \|r_l\|_\cD \le \hat\theta :=
    \frac{\theta(1+\mu)}{1-\mu-2\mu^+}
    $$    
    has the following properties:
    \begin{itemize}  
        \item[(i)] it selects only ``true'' dictionary elements, i.e. from $\{g_k\}_{k=1}^K$;
        \item[(ii)] it selects all elements $g_k$ with the coefficients
            $|c_k|>\hat\beta$.
    \end{itemize}
\end{Theorem}

 In Subsection \ref{sub_channel} we describe  the channel estimation problem from wireless communications  and show how to apply Theorem \ref{thm_mu_plus} in this situation. In fact, Theorem \ref{thm_mu_plus} is designed to deal with the oversampled Fourier matrix -- a typical example of a matrix with coherence close to $1.$     

 Motivated by applications in wireless communications  we study in Sections \ref{FD} and \ref{Disc} sparse approximation of elements, which have expansions with exponentially fast 
decaying coefficients. We prove  Theorem \ref{FDT1}  for a  general setting  (a  Banach space $X$  with a regular modulus of smoothness, a dictionary  $\cD$)  that for any element $f= \sum_{k=0}^\infty a_k g_k,\quad g_k\in \cD,\quad |a_k| \le q^k,\quad k=0,1,\dots,\quad q\in (0,1)$ we have $
\sigma_m(f,\cD) \le C q^m (1-q)^{-1/2}
$, and establish that it provides the best possible dependence on $q$ in the bound. 

\section{Definitions of greedy algorithms}
\label{algor}

We define the Pure Greedy Algorithm (PGA). We describe this algorithm for a general dictionary $\cD$. If $f\in \mathcal{H}$,
we let $g(f)\in \cD$ be an element from $\cD$ which maximizes $|\< f,g\>|$. We assume for simplicity that such a maximizer exists; if not suitable modifications are necessary (see Weak Greedy Algorithm in \cite{VTbook}) in the algorithm that follows. We define
$$
G(f,\cD):= \<f,g(f)\>g(f)\quad \text{and}\quad R(f,\cD) := f-G(f,\cD).
$$
 
 {\bf Pure Greedy Algorithm (PGA).} We define $f_0:= f$ and $G_0(f,\cD) := 0$. Then, for each $m\ge 1$, we inductively define
$$
G_m(f,\cD):= G_{m-1}(f,\cD) +G(f_{m-1},\cD)
$$
$$
f_m:= f-G_m(f,\cD) = R(f_{m-1},\cD).
$$
Note that for a given element $f$ the sequence $\{G_m(f,\cD)\}$ may not be unique. 

This algorithm is well studied from the point of view of convergence and rate of convergence. The reader 
can find the corresponding results and historical comments in \cite{VTbook}, Ch.2.

  Let a sequence $\tau = \{t_k\}_{k=1}^\infty$,
$0\le t_k \le 1$, be given. The following greedy algorithm was
defined in \cite{T1} under the name Weak Orthogonal Greedy Algorithm (WOGA).

 {\bf Weak Orthogonal Matching Pursuit (WOMP).} Let $f_0$ be given. Then for each $m\ge 1$ we inductively define:

(1) $\psi_m  \in \cD$ is any element satisfying
$$
|\langle f_{m-1},\psi_m\rangle | \ge t_m
\sup_{g\in \cD} |\langle f_{m-1},g\rangle |.
$$

(2) Let $\mathcal{H}_m := \sp (\psi_1,\dots,\psi_m)$ and let
$P_{\mathcal{H}_m}(\cdot)$ denote an operator of orthogonal projection onto $\mathcal{H}_m$.
Define
$$
G_m(f_0,\cD) := P_{\mathcal{H}_m}(f_0).
$$

(3) Define the residual after $m$th iteration of the algorithm
$$
f_m := f_0-G_m(f_0,\cD).
$$

In the case $t_k=1$, $k=1,2,\dots$,   WOMP is called the Orthogonal
Matching Pursuit (OMP). In this paper we only consider the case $t_k=t$, $k=1,2,\dots$, $t\in(0,1]$.  The term {\it weak} in the definition of the WOMP means that at step (1) we do not shoot for the optimal element of the dictionary which realizes the corresponding supremum but we are satisfied with a weaker property than being optimal. The obvious reason for this is that we do not know in general that the optimal element exists. Another practical reason is that the weaker the assumption the easier it is to satisfy it and, therefore, easier to realize in practice. Clearly, $\psi_m$ may not be unique. Our results apply for any realization (any choice of $\psi_m$) of the WOMP.

  \section{Performance of the WOMP for  Riesz dictionaries}
 \label{O}
 
 
 Let us begin our discussion with the case, when we assume that our clean signal $f^*$ is sparse: $f^* \in \Sigma_m(\cD)$ with respect to a given dictionary $\cD$. We recall that 
$
\sigma_m(f,\cD) := \inf_{g\in \Sigma_m(\cD)} \|f-g\|.
$
Assume that we have the following inequality, which is called the Lebesgue-type inequality, for the WOGA($t$)
\be\label{Leb1}
\|f_m\| \le C(t)\sigma_m(f,\cD). 
\ee
Then (\ref{Leb1}) guarantees that under assumption (\ref{In1}) after $m$ iterations of the  WOGA($t$) applied to the noisy version $f$ of the clean signal 
$f^*$ we obtain
$$
\|f^* - G_m(f,\cD)\| \le \|f-f^*\| + \|f_m\| \le (1+C(t))\epsilon.
$$
This indicates a direct connection between approximation of noisy data and the Lebesgue-type inequalities. We now cite some known results on the Lebesgue-type inequalities. 

It was pointed out in \cite{LivTem}  that results on the Lebesgue-type inequalities proved in the Banach spaces case provide a corollary for Hilbert spaces that gives sufficient conditions somewhat weaker than the known RIP conditions on $\cD$ for the Lebesgue-type inequality to hold.

The following concept is useful for general  Banach spaces. 
 
  {\bf UP($v,D$). ($v,D$)-unconditional property.}  We say that a dictionary $\cD =\{g_i\}_{i=1}^\infty$ is ($v,D$)-unconditional with a constant $U$ if for any $v$-sparse
   element $f=\sum_{i\in A}x_ig_i$ ($|A|\le v$)  and any $\Lambda$, such that $A\cap \Lambda =\emptyset$ and $|A|+|\Lambda| \le D$, we have for any $\{c_i\}$
\be\label{UP}
\left\|\sum_{i\in A}x_ig_i-\sum_{i\in\Lambda}c_ig_i\right\|_X\ge U^{-1}\left\|\sum_{i\in A}x_ig_i\right\|_X.
\ee
 
 We now formulate a result from \cite{LivTem} (see also \cite{VTbookMA}, Section 8.7). Theorem \ref{ssrT1} is a development of the breakthrough result by Zhang (see \cite{Z}). 
 
  \begin{Theorem}[{\cite[Corollary I.1]{LivTem}}]\label{ssrT1} Let $\mathcal{H}$ be a Hilbert space. Suppose that the dictionary $\cD$ has property  {\bf UP($v,D$)} with a constant $U$.   Then there exists a constant $c_{\ref{ssrT1}}(t,U)$ such that for any $f_0\in \mathcal{H}$ the WOMP with weakness parameter $t$ applied to $f_0$ provides
$$
\|f_{c_{\ref{ssrT1}}(t,U) v}\| \le C\sigma_v(f_0,\cD)\quad\text{for}\quad v+c_{\ref{ssrT1}}(t,U) v\le D
$$
with an absolute constant $C$.
\end{Theorem}

\subsection{Sampling recovery. Proof of Theorem \ref{IT3}}
\label{ssr}

In this section we focus on special collections of subspaces generated by a given dictionary. 
Let   $\D_N=\{g_i\}_{i=1}^N$ be a given dictionary. 
Given an integer $1\leq v\leq N$, we denote by $\mathcal{X}_v(\D_N)$ the collection of all linear spaces spanned by $g_j$, $j\in J$  with $J\subset [1,N]\cap \N$ and $|J|=v$, and 
denote by $\Sigma_v(\D_N)$  the set of all $v$-term approximants with respect to $\D_N$:
\begin{align*}
\Sigma_v(\D_N):&= \bigcup_{V\in\cX_v(\D_N)} V.
\end{align*}
Denote
$$
\Sigma_v^q(\D_N):= \{f\in \Sigma_v(\D_N)\,:\, \|f\|_q\le 1\}.
$$

   \begin{Definition}\label{ID1} We say that a set $\xi:= \{\xi^j\}_{j=1}^m \subset \Omega $ provides {\it   universal discretization}   for the collection $\cX:= \{X(n)\}_{n=1}^k$ of finite-dimensional  linear subspaces $X(n)$ if we have
 \be\label{I3}
\frac{1}{2}\|f\|_2^2 \le \frac{1}{m} \sum_{j=1}^m |f(\xi^j)|^2\le \frac{3}{2}\|f\|_2^2\quad \text{for any}\quad f\in \bigcup_{n=1}^k X(n) .
\ee
We denote by $m(\cX)$ the minimal $m$ such that there exists a set $\xi$ of $m$ points, which
provides  universal discretization (\ref{I3}) for the collection $\cX$. 
\end{Definition}

We now formulate Theorem \ref{IT2}, which provides three Lebesgue-type inequalities and after that explain some notations 
used in its formulation. We are interested in the error in the $L_2(\Og,\mu)$. Along with the space $L_2(\Omega,\mu)$ we consider the space $L_2(\Omega_m,\mu_m)$
where $\Omega_m=\{\xi^\nu\}_{\nu=1}^m$   and  $\mu_m(\xi^\nu) =1/m$, $\nu=1,\dots,m$. Let $\Phi_N(\Omega_m)$ be the restriction 
of $\Phi_N$ onto $\Omega_m$.

\begin{Theorem}\label{IT2}  Let a natural number $s\le N$ be given.   Assume that $\Phi_N$ belongs to $\mathcal{R}(s,R_1,R_2)$ for some constants $0<R_1\leq R_2<\infty$. Also assume that $\Phi_N$ is a uniformly bounded system satisfying  (\ref{I9}).   For given parameters $t\in (0,1]$ and $R_1$, $R_2$ from above there exists a constant $c=C(t,R_2/R_1)$ with the following property. Let for an integer $v$ such that $u:=(1+c) v\leq s$ and a number $m\ge m(\cX_u(\Phi_N))$ the set $\Og_m$ of $m$ points
$\xi^1,\cdots, \xi^m \in  \Omega$ provide   universal discretization for the collection 
$\cX_u(\Phi_N)$.  Then for any $f_0\in \cC(\Omega)$ the WOMP with weakness parameter $t$ applied to $f_0$ with respect to the   $\Phi_N(\Omega_m)$ in the space $L_2(\Omega_m,\mu_m)$ provides
\be\label{mp}
\|f_{c v}\|_{L_2(\Omega_m,\mu_m)} \le C\sigma_v(f_0,\Phi_N(\Omega_m))_{L_2(\Omega_m,\mu_m)}, 
\ee
and
\be\label{mp2}
\|f_{c v}\|_{L_2(\Omega,\mu)} \le C'\sigma_v(f_0,\Phi_N)_\infty 
\ee
with absolute constants $C$ and $C'$.
Moreover, we have
 \be\label{mp3}
  \|f_{cv} \|_{L_2(\Omega,\mu)} \le C'' \sigma_v(f_0,\Phi_N)_{L_2(\Og, \mu_\xi)}
 \ee
 with an absolute constant $C''$ and 
 $$
\mu_\xi := \frac{1}{2} \mu + \frac{1}{2m} \sum_{j=1}^m \delta_{\xi^j}.
$$
 \end{Theorem}
 
  Note, that the most difficult assumption in Theorem \ref{IT2} is the assumption that the set $\Og_m$ of $m$ points $\xi^1,\cdots, \xi^m \in  \Omega$ provides   universal discretization for the collection $\cX_u(\Phi_N)$. 
There are two difficult question in that regard. The first question is the existence of such a set with good $m$. In other words,
it is a question of finding a good upper bound for the $ m(\cX_u(\Phi_N))$. The second question is how to construct such a good set. Both of these questions were intensely studied for the last decade. The reader can find some material in the surveys \cite{DPTT}, \cite{DKT}, \cite{KKLT}, \cite{LMT}. It turns out that the probabilistic approach gives rather 
accurate results.  

 \begin{proof}[Proof of Theorem \ref{IT2}]  We begin with a simple Proposition \ref{InP1}.

\begin{Proposition}\label{InP1}
  Any $\cD \in \mathcal{R}(u,R_1,R_2)$ satisfies UP($v,u$)  with a constant $U=R_2/R_1$ for any $v\le u$. 
\end{Proposition}
\begin{proof} Indeed, we have 
$$
\left\|\sum_{i\in A}x_ig_i-\sum_{i\in\Lambda}c_ig_i\right\|_\mathcal{H}\ge R_1 \left(\sum_{i\in A}|x_i|^2 + \sum_{i\in\Lambda}|c_i|^2\right)^{1/2} 
$$
$$
\ge R_1 \left(\sum_{i\in A}|x_i|^2  \right)^{1/2}  \ge (R_1/R_2)\left\|\sum_{i\in A}x_ig_i \right\|_\mathcal{H}.
$$
\end{proof}

Proposition \ref{InP1} implies that the system $\Phi_N$ has the {\bf UP}$(v,u)$ in the $L_2(\Omega,\mu)$ with any parameters    $v\le u$ and the constant $U_1 = R_2/R_1$. Our assumption that the set $\Og_m$ of $m$ points
$\xi^1,\cdots, \xi^m \in  \Omega$ provides   universal discretization for the collection 
$\cX_u(\Phi_N)$ (see inequalities 
(\ref{I3})) impies that the discretized system $\Phi_N(\Omega_m)$ has the {\bf UP($v,u$)} in the $L_2(\Omega_m,\mu_m)$ with the constant $U_2 =U_1 3^{1/2}$. 

We now define a dictionary $\D_N$ in the $L_2(\Omega_m,\mu_m)$ as the normalized system $\Phi_N(\Omega_m)$. Clearly, property {\bf UP} does not depend on the normalization. 
We plan to apply the WOMP algorithm with the weakness parameter $t$. We now set $c=c_{3.1}(t,U_2)$ and for a given $v$  set $u=v+c_{3.1}(t,U_2)v$ with the constant $c_{3.1}(t,U_2)$ from Theorem \ref{ssrT1}. 
It remains to apply Theorem \ref{ssrT1}. This proves (\ref{mp}) in Theorem \ref{IT2}. 

We now derive (\ref{mp2}) from (\ref{mp}).   Clearly, 
$$
\sigma_v(f_0,\Phi_N(\Omega_m))_{L_2(\Omega_m,\mu_m)} \le \sigma_v(f_0,\Phi_N )_\infty.
$$
Let $f\in \Sigma_v(\Phi_N)$ be such that  $\|f_0-f\|_\infty \le 2 \sigma_v(f_0,\Phi_N)_\infty$. Then (\ref{mp}) implies 
$$
\|f - G_{cv}(f_0,\Phi_N(\Omega_m))\|_{L_2(\Omega_m,\mu_m)} \le \|f-f_0\|_{L_2(\Omega_m,\mu_m)} +\|f_{cv}\|_{L_2(\Omega_m,\mu_m)} 
$$
$$
\le (2+C)\sigma_v(f_0,\Phi_N)_\infty.
$$
Using that $f - G_{cv}(f_0,\Phi_N(\Omega_m)) \in \Sigma_u(\Phi_N)$, by discretization (\ref{I3}) we 
conclude that
\be\label{ssr3}
\|f - G_{cv}(f_0,\Phi_N(\Omega_m))\|_{L_2(\Omega,\mu)} \le 2^{1/2}(2+C)\sigma_v(f_0,\Phi_N)_\infty.
\ee
Finally,
$$
\|f_{cv}\|_{L_2(\Omega,\mu)} \le \|f-f_0\|_{L_2(\Omega,\mu)} + \|f - G_{cv}(f_0,\Phi_N(\Omega_m))\|_{L_2(\Omega,\mu)}.
$$
This and (\ref{ssr3}) prove (\ref{mp2}).

We now prove (\ref{mp3}).  It is clear that the space $L_2(\Omega_m,\mu_m)$ can be seen as the $L_2(\Omega,\mu_m)$ with 
 $$
 \mu_m = \frac1{m}\sum_{j=1}^m \delta_{\xi^j}.
 $$
 Thus, we have $\mu_\xi = (\mu+\mu_m)/2$.
 We  use discretization  assumption (\ref{I3}) (we only need the left inequality in (\ref{I3})) and the inequalities
$$
\|h\|_{L_2(\Og, \mu)}	\le 2^{1/2} \|h\|_{L_2(\Og,\mu_{\xi})},
$$
$$
\|h\|_{L_2(\Omega_m,\mu_m)}=\|h\|_{L_2(\Omega,\mu_m)} \le 2^{1/2}\|h\|_{L_2(\Omega,\mu_\xi)}.
$$
Then
$$
\sigma_v(f_0,\Phi_N(\Omega_m))_{L_2(\Omega_m,\mu_m)} \le 2^{1/2}\sigma_v(f_0,\Phi_N )_{L_2(\Omega,\mu_\xi)}.
$$
Let $f\in \Sigma_v(\Phi_N)$ be such that  $\|f_0-f\|_{L_2(\Omega,\mu_\xi)} \le 2 \sigma_v(f_0,\Phi_N)_{L_2(\Omega,\mu_\xi)}$. Then (\ref{mp}) implies 
$$
\|f - G_{cv}(f_0,\Phi_N(\Omega_m))\|_{L_2(\Omega_m,\mu_m)} \le \|f-f_0\|_{L_2(\Omega_m,\mu_m)} +\|f_{cv}\|_{L_2(\Omega_m,\mu_m)} 
$$
$$
\le 2^{1/2}(2+C)\sigma_v(f_0,\Phi_N)_{L_2(\Omega,\mu_\xi)}.
$$
Using that $f - G_{cv}(f_0,\Phi_N(\Omega_m)) \in \Sigma_u(\Phi_N)$, by discretization (\ref{I3}) we 
conclude that
\be\label{ssr3'}
\|f - G_{cv}(f_0,\Phi_N(\Omega_m))\|_{L_2(\Omega,\mu)} \le 2 (2+C)\sigma_v(f_0,\Phi_N)_{L_2(\Omega,\mu_\xi)}.
\ee
Finally,
$$
\|f_{cv}\|_{L_2(\Omega,\mu)} \le \|f-f_0\|_{L_2(\Omega,\mu)} + \|f - G_{cv}(f_0,\Phi_N(\Omega_m))\|_{L_2(\Omega,\mu)}.
$$
This and (\ref{ssr3'}) prove (\ref{mp3}).

\end{proof}

 Theorem \ref{IT2} has a very important assumption that the set $\Og_m$ of $m$ points
$\xi^1,\cdots, \xi^m \in  \Omega$ provides   universal discretization for the collection 
$\cX_u(\Phi_N)$. We now discuss a way of satisfying this assumption. The following Theorem \ref{IT1} was proved in \cite{DTM2} (see also \cite{DKT}). 

\begin{Theorem}[\cite{DKT}, Theorem 5.2]\label{IT1}  Let $ \cD_N=\{\ff_j\}_{j=1}^N$ be a set of $N$ uniformly bounded functions on $\Omega$
satisfying
$$
\max_{1\leq j\leq N} \|\ff_j\|_{L_\infty(\Omega)}\le 1.
$$
Let $1\leq s\leq N$ be a given integer. Assume that there exists a constant $K\ge 1$ such that
\begin{equation}\label{Bessel}
\sum_{j\in J} |a_j|^2 \le K  \Bigl\|\sum_{j\in J} a_j\ff_j\Bigr\|^2_{L_2(\Omega, \mu)},\   \   \ \forall a_j\in\bbC
\end{equation}
whenever  $J\subset \{1,2,\ldots, N\}$ with $|J|=s$.
Let $\xi^1,\ldots, \xi^m$ be iid
random points  with common  distribution  $\mu$ on $\Omega$.
Then given  $1\le p\le 2$ and $\varepsilon\in (0,\frac12]$,
there exist constants $C_p(\varepsilon)>1$ and $c_p(\varepsilon)>0$,
{depending only on $p$ and $\varepsilon$}, such that the inequalities
\be\label{pdisc}
(1-\varepsilon)\|f\|_{L_p(\Omega, \mu)}^p \le \frac{1}{m}\sum_{j=1}^m |f(\xi^j)|^p \le (1+\varepsilon)\|f\|_{L_p(\Omega, \mu)}^p,\   \   \ \forall f\in  \Sigma_s(\cD_N)
\ee
hold with probability $\ge 1-2 \exp\left( -\frac {c_p(\varepsilon) m}{Ks\log^2 (2Ks)}\right)$ provided that
$$
m \ge  C_p(\varepsilon) Ks \log N\cdot (\log(2Ks ))^2\cdot (\log (2Ks )+\log\log N).
$$
\end{Theorem}

 Note that our assumption that $\Phi_N$ belongs to $\mathcal{R}(s,R_1,R_2)$ implies assumption (\ref{Bessel}) in Theorem \ref{IT1} with $K= R_1^{-2}$.
We now apply Theorem \ref{IT1} with $p=2$, $\varepsilon =1/2$ and obtain that with a high probability the assumption of the universal discretization in Theorem \ref{IT2} follows from the assumption that the system $\Phi_N$ is uniformly bounded and   $\Phi_N\in \mathcal{R}(s,R_1,R_2)$. 

As a result of the above discussion we obtain Theorem \ref{IT3} formulated in the Introduction.

 \begin{Remark}\label{InR1a} For a fixed $K$ Theorem \ref{IT3} works for 
 $$
 m\gg s (\log N) (\log s )^2\cdot (\log s +\log\log N).
 $$ 
 We obtain this bound from Theorem \ref{IT1} on the universal discretization, which holds for the $L_p$ norms, $1\le p\le 2$.
 We only use the case $p=2$.  It is known (see \cite{DTM2}) that the problem of universal discretization of the $L_2$ norm 
 is equivalent to the study of the RIP problem. This and recent results on the RIP problem (see \cite{HR} and \cite{BDJR}) allow us to improve the bound for $m$ in Theorem \ref{IT3} to $m\gg s (\log N) (\log s )^2 $ under an extra assumption that the system $\Phi_N$ is the Riesz basis.
 \end{Remark}

\subsection{Finite-dimensional case and noisy values}
\label{subsec_finite_dim}

In practical applications, e.g. the sparse channel recovery problem
(see Subsection \ref{sub_channel}), signals are complex vectors in some
finite-dimensional space. We discuss here an application of the above theoretical
results and explain the main ingredients of the proofs in this
finite-dimensional setting.

One can identify vectors in $\bbC^M$ with functions defined on the domain
  $\Omega=\{1,\ldots,M\}$. The measure $\mu$
is defined  as  $\mu(\bx^i)=1/M$, $i=1,\dots,M$. 
The $i$-th coordinate of a vector $f\in\mathbb{C}^M$ is denoted as $f(i)$.
The standard euclidean norm 
$$
\|f\|_2=(\sum_{i=1}^M |f(i)|^2)^{1/2}
$$
is considered. Note that the $L_2(\Omega,\mu)$-norm of $f$ equals
$M^{-1/2}\|f\|_2$ due to the normalization.

\paragraph{Noisy values.}

Given a sequence $I_m=(i_1,\ldots,i_m)$ of coordinates, we denote
$f[I_m]:=(f(i_1),\ldots,f(i_m))\in\bbC^m$.

Let $\Phi\subset\bbC^M$ be a dictionary and $f^*\in\Sigma_v(\Phi)$ be a clean
signal. Consider noisy values
$$
y_k = f^*(i_k) + w_k,
\quad k=1,\ldots,m,
$$
where $w\in\bbC^m$ is some noise. An equivalent notation: $y = f^*[I_m] + w$. Our
goal is to recover $f^*$ using vector $y$.

This is a particular case of the problem we considered before: recover
$v$-sparse (with respect to a dictionary) function $f^*\colon
\Omega\to\bbC$ given the values $\left.f\right|_{\Omega_m}$, where
$\Omega_m:=\{\bx^{i_1},\ldots,\bx^{i_m}\}$. (Here we assume for simplicity that
all $i_k$ in $I_m$ are distinct.)
Besides, we have the additional condition that the ``noise''
is concentrated on the set $\Omega_m$:
$$
w(\bx):=f(\bx)-f^*(\bx)=0\quad \mbox{for} \quad \bx \in \Omega \setminus \Omega_m.
$$

Let us return to the vector case.
The solution of the problem is divided into two steps. First, we recover $f^*[I_m]$ by
constructing an approximation
\begin{equation}
    \label{f_approx}
\hat y = \sum_{k=1}^l \hat c_k g_k[I_m],\quad g_k \in \Phi.
\end{equation}
Second, we guarantee that if
$\hat y$ approximates
$f^*[I_m]$ then $\hat f = \sum_{k=1}^l \hat c_k g_k$ approximates $f^*$.

\paragraph{Riesz dictionaries.}
We will keep the normalization condition
$\|g\|_\infty\le 1$, so it is reasonable to formulate the Riesz property of a
dictionary $\Phi\subset\bbC^M$ in the following way:
\begin{equation}
    \label{riesz_discrete}
R_1 \|a\|_2
    \le \frac{1}{\sqrt{M}}\left\|\sum_{k=1}^s a_k g_k\right\|_2
\le R_2 \|a\|_2,
\quad\forall\,g_1,\ldots,g_s\in\Phi_N,\; a\in\mathbb C^s.
\end{equation}

Theorem~\ref{ssrT1} of Livshitz and Temlyakov guarantees that any such dictionary $\Phi$
with bounded $R_2/R_1$ ratio allows good recovery in the following way. If $v+cv\le s$, then for
any $f^*\in\Sigma_v(\Phi)$ the OMP algorithm with the initial vector $f=f^*+w$
after $cv$ steps returns the approximation
$\hat f$ that satisfies
$\|\hat f - f^*\|_2 \le C \|w\|_2$. We use this in the first step.

\paragraph{Discretization.}
Let us repeat the definition of the universal discretization in
finite-dimensional setting. A sequence $I_m=(i_1,\ldots,i_m)$ of
coordinates provides the discretization for a set of vectors
$V\subset\bbC^M$, if
$$
\frac{1}{2M}\|f\|_2^2
\le \frac{1}{m} \|f[I_m]\|_2^2
\le \frac{3}{2M} \|f\|_2^2
\quad \mbox{for any $f\in V$}.
$$

If the discretization condiction holds for the set $V=\Sigma_s(\Phi)$, and
$l+v\le s$, then we clearly have
(see~\eqref{f_approx})
$$
\frac1M\|f^*-\hat f\|_2^2
\le \frac{2}{m}\|f^*[I_m]-\hat y\|_2^2.
$$
This is what we need for the second step.

Theorem~\ref{IT1} of Dai and Temlyakov bounds the number of coordinates
$m$ used for the discretization. Note that this bound does not depend on $M$.

Now let us formulate a precise statement.

\begin{Corollary}
Let $\Phi_N\subset\mathbb{C}^M$ be a set of $N$ vectors, such that
    $\|g\|_\infty\le 1$ for all $g\in\Phi_N$, and the
    property~\eqref{riesz_discrete} is true with $R_2/R_1\le C_0<\infty$.
Let $I_m:=(i_1,\ldots,i_m)$ be a sequence of $m$ i.i.d. coordinates in
    $\{1,\ldots,M\}$, the number $m$ satisfies the
    condition~\eqref{IT3_m}.
    Then with high probability  the following holds.

    Consider any vector $f^*=\sum_{k=1}^v c_k g_k$, $g_k\in\Phi_N$, where
    $(1+c)v\le s$.
    Run the OMP algorithm with the initial vector $y:=f^*[I_m]+w$, where
    $w\in\bbC^m$ is some vector  (noise), and the dictionary
    $\Phi_N[I_m] := \{g[I_m]\colon g\in\Phi_N\}$.
Let $\hat y=\sum_{k\le cv}\hat c_k g_k[I_m]$ be the
approximation obtained after $cv$ steps.
Then the following inequality holds
$$
    \frac1M\|f^*-\hat f\|_2^2 \le  \frac{C}m\|w\|_2^2,
    \quad\mbox{where $\hat f := \sum_{k\le cv} \hat c_k g_k$}.
$$
\end{Corollary}

\begin{proof}
    Define a vector $f$ as follows: $f(i):=f^*(i)$ when
    $i\not\in\{i_1,\ldots,i_m\}$,
    and $f(i_k):=y_k$, $k=1,\ldots,m$.
    Here we assume for simplicity that any
    index $i_k$ appears in $I_m$ only once.
    We apply  inequality (\ref{Inmp3}) of  Theorem~\ref{IT3} to the vector $f$, the dictionary $\Phi_N$ and the
    probability measure $\mu:=(1/M)\sum_{i=1}^M\delta_i$ and get that
    \begin{multline*}
    \frac1M\|f-\hat f\|_2^2
    = \|f-\hat f\|_{L_2(\mu)}^2
    \lesssim \sigma_v(f,\Phi_N)_{L_2(\mu_\xi)}^2
    \le \|f-f^*\|_{L_2(\mu_\xi)}^2 = \\
    = \frac12 \|f-f^*\|_{L_2(\mu)}^2 + \frac12\|f-f^*\|_{L_2(\mu_m)}^2
        \le \frac1{2M} \|w\|_2^2 + \frac1{2m} \|w\|_2^2.
    \end{multline*}

    Moreover, $\|f^*-\hat f\|_2 \le \|f^*-f\|_2 + \|f-\hat f\|_2 \le \|w\|_2 + \|f-\hat
    f\|_2$. The combination of these bounds gives the required result.

    Finally, let us return to the technical subtlety: if an index $i$ appears in $I_m$ more
    than once, we define $f(i)$ as the average of all corresponding $y_k$;
    hence, $f(i)=f^*(i)+\mathrm{avg}\{w_k\colon i_k=i\}$.
    It is easy to check that $\|f-f^*\|_2 \le \|w\|_2$ in this case, too.
    Note that the OMP applied to the vector $y$ in the space $\mathbb{C}^m$ and
    the dictionary $\Phi_N[I_m]$ is equivalent to the OMP
    applied to $f$ and the dictionary $\Phi_N$ in the space $L_2(\mu_m)$,
    $\mu_m:=(1/m)\sum_{j=1}^m\delta_{i_j}$. So the proof remains correct.
\end{proof}

\section{Recovery with coherent dictionaries}
\label{FD}
Theorem \ref{FDT1} (see below) gives upper bounds for the best $m$-term
approximation and Theorem \ref{FDT1} holds in a rather general situation.  We
now give a comment on performance of a greedy algorithm. We assume that $\mathcal{H}$ is
a Hilbert space and a dictionary $\cD$ satisfies some special conditions. We
discuss the OMP greedy algorithm. It is known (see, for instance,
\cite{VTbookMA}, Section 8.7.2)  that under some conditions on the dictionary
$\cD$ the OMP provides the Lebesgue-type inequality \be\label{LebI}
\|f_{cK}\| \le C\sigma_K(f,\cD)
\ee
with absolute constant $C$ and constant $c$, which may depend on 
characteristics of the dictionary $\cD$. As the corresponding conditions on the
dictionary $\cD$ the RIP and the somewhat more general condition {\bf A2} (see
\cite{VTbookMA}, p.426) can be taken. Then, Theorem \ref{FDT1} and the Lebesgue
inequality (\ref{LebI}) give a very good upper bound on the norm of the
residual. However, it might not be convenient to use this way in practical
applications because in many practical problems the dictionaries have large RIP
constants and even the coherence may be close to $1$.

A typical  example of a coherent dictionary is the oversampled Fourier dictionary.
The oversampling occurs in different situations; the one particular that motivated us
is the problem of sparse channel recovery in wireless communication.
In Subsection~\ref{sub_mu_plus} we give a theoretical result for the OMP
algorithm with coherent dictionaries and in Subsection~\ref{sub_channel} we
apply this result to the sparse channel recovery problem. In
Subsection~\ref{sub_known_coherence}, we briefly discuss some known results on
OMP recovery with incoherent dictionaries.

In Subsection \ref{ss} we discuss an example
also oriented to practical applications
--- the sparse approximation of elements, which have expansions with exponentially fast 
decaying coefficients.

\subsection{Some known results for OMP with coherent dictionaries}
\label{sub_known_coherence}

We cite two theorems on OMP recovery with dictionaries that have small
coherence. The first one is~\cite[Theorem 5.1]{DET}.

\begin{OldTheorem}
    Suppose that $\mathcal D$ is a dictionary with $M(\cD)\le \mu$, the
    clean signal $y_0$ is $K$-sparse and we observe a noisy signal
    $$
    y = y_0 + w,
    \quad y_0 = \sum_{k=1}^K c_k g_k,\quad g_k\in\cD,
    \quad\|w\|\le\e.
    $$

    If the inequality holds:
    $$
    (2K-1)\mu + \frac{2\e}{\min_k|c_k|} \le 1,
    $$
    then the OMP algorithm with the dictionary $\cD$, the initial vector $y$ and
    the stopping criterion $\|r_l\|_2\le\e$ has the following properties:
    \begin{itemize}
        \item it recovers correctly all elements $g_k$, $k=1,\ldots,K$, and only
            them;
        \item the approximation error for the coefficients satisfies
            $$
            \|\hat c - c\|_2^2 \le \e^2(1-\mu(K-1))^{-1}.
            $$
    \end{itemize}
\end{OldTheorem}

The next result is~\cite[Theorem 5.3]{TGS06}. 
Instead of the $\ell_2$-norm of the noise $w$ it uses its max-correlation
norm~\eqref{maxcorr_norm}. The difference is important in some settings, e.g.
when the noise is Gaussian (see Subsection~\ref{sub_channel}).
We cite a corollary of this theorem (see also~\cite[Theorem 18]{Tr06}).

\begin{OldTheorem}
\label{th_tropp}
    Suppose that $\mathcal D$ is a dictionary with $M(\mathcal D)\le \mu$, the
    input signal $y$ has the best $K$-term approximation $y_0$ such that
    $$
    y=y_0+w,
    \quad y_0 = \sum_{k=1}^K c_kg_k,
    \quad \|w\|_{\cD}\le \gamma\frac{1-2K\mu}{1-K\mu}.
    $$
    Then the OMP algorithm with the dictionary $\cD$, the initial vector $y$ and
    the stopping criterion $\|r_l\|_\cD\le \gamma$ has the following properties:
    \begin{itemize}
        \item it selects only elements from the set $\{g_k\}_{k=1}^K$;
        \item it selects all $g_k$ with the coefficients
            $|c_k|>\gamma(1-2K\mu)^{-1}$.
    \end{itemize}
\end{OldTheorem}

\subsection{Recovery guarantees for coherent dictionaries}
\label{sub_mu_plus}
We recall the definition of $M^+(\cD_1,\cD)$ quantity from the Introduction. 
Namely:
$
M^+(\cD_1,\cD) := \max_{g\in\mathcal D} \left\{\sum_{h\in\mathcal
D_1}|\langle g,h\rangle| - \max_{h\in\mathcal D_1}|\langle g,h\rangle|\right\}.
$

From the definition we see that for any $g\in\cD$ and $h_0\in\cD_1$ we have
\begin{equation}
    \label{mu_plus_cor1}
\sum_{h\in\cD_1}|\langle g,h\rangle|
\le M(\cD) + M^+(\cD_1,\cD),
\end{equation}
\begin{equation}
    \label{mu_plus_cor2}
\sum_{h\in\cD_1,\;h\ne h_0}|\langle h,h_0\rangle|
\le M^+(\cD_1,\cD).
\end{equation}

We proceed to the proof of Theorem \ref{thm_mu_plus}
\begin{proof} 
    The selected thresholds satisfy the equalities:
    \begin{equation}
        \label{corr_equality}
    \hat\beta(1-\mu^+) - \theta = \hat\beta(\mu + \mu^+) + \theta = \hat\theta.
    \end{equation}

    Let us call a coeffient $c_k$ \emph{large}, if $|c_k|>\hat\beta$.

    We carry out induction on the step number. After step $l-1$ we have
    $$
    r_{l-1} = y - \sum_{j=1}^{l-1} c_j^* g_j^*,
    $$
    where  all selected elements $g_j^*\in\{g_k\}_{k=1}^K$ (by the induction hypothesis). So the
    remainder equals
    $$
    r_{l-1} = \sum_{k=1}^K c_k'g_k + w,
    $$
    where the coeffient $c_k'$ either equals $c_k$, or differs from it, and
    the latter case occurs only when we have selected element $g_j^*=g_k$ on some
    previous step. Denote $c^*:=\max|c_k'|$ and suppose that the maximum is
    attained at $|c'_{k_0}|=c^*$.

    Suppose that the current set of the coefficients contains an ``untouched'' large
    coefficient. Then $c^*>\hat\beta$. We prove that in this case we will choose
    some of the elements $\{g_k\}_{k=1}^K$ during the $l$-th step.
    Indeed, using the definition of $\mu^+$, we obtain:
    $$
    |\langle g_{k_0},r_{l-1}\rangle|
    \ge c^* - c^*\sum_{\substack{1\le k\le K\\ k\ne k_0}}|\langle g_{k_0},g_k\rangle| - |\langle g_{k_0},w\rangle|
    \ge c^*(1-\mu^+) - \theta.
    $$
    On the other hand, for any $g\not\in\{g_k\}_{k=1}^K$ we have (in the same
    way as in~\eqref{mu_plus_cor1}):
    $$
    |\langle g,r_{l-1}\rangle| \le c^*\sum_{k=1}^K |\langle g,g_k\rangle| + |\langle
    g,w\rangle| \le c^*(\mu + \mu^+) + \theta.
    $$
    Let us compare the lower and the upper bounds. If $c^*=\hat\beta$
    then the bounds become equal (see~\eqref{corr_equality}), so in our case
    $c^*>\hat\beta$ the first (lower) bound is greater than the second.
    Hence, all ``incorrect'' dictionary elements have smaller correlation and we
    will select only from ``true'' elements.

    Moreover, we have proved that if there is at least one untouched large
    coefficient then the stopping criteria is not satisfied~--- there is a
    correlation at least
    $$
    c^*(1-\mu^+)-\theta > \hat\beta(1-\mu^+) - \theta = \hat\theta,
    $$
    see~\eqref{corr_equality}. Hence, we will find all elements with large coeffients.

    If there are no large coefficients, then $c^*\le \hat\beta$ and our
    reasoning does not work. But if we suppose that the maximum of $|\langle
    g,r_l\rangle|$ is attained at an element $g\not\in\{g_k\}$,
    then the correlation is at most
    $$
    c^*(\mu+\mu^+)+\theta \le \hat\beta(\mu+\mu^+)+\theta= \hat\theta
    $$
    and our algorithm stops (and ignores the incorrect element).
\end{proof}

\begin{Corollary}
    If all the coefficients are large, i.e.
    $$
        \mu + 2\mu^+ + \frac{2\theta}{\min_k|c_k|} < 1,
    $$
    then OMP stops after $K$ steps and finds all
    $g_k$, $k=1,\ldots,K$.
\end{Corollary}

\subsection{Sparse channel recovery problem in  wireless communication}
\label{sub_channel}
We very briefly describe the channel estimation problem and show what follows from Theorem \ref{thm_mu_plus}. The reader can read about wireless communications in \cite{App11}, \cite{HL18},\cite{XRS},\cite{Zh16}.   
In our setting  the unknown clean signal is  the frequency response of a sparse channel. It is given by the formula
$$
H_n = \sum_{k=1}^K \mathfrak{c}_k \exp(-2\pi i n\tau_k/T),
\quad n\in\Omega:=\{0,1,\ldots,N-1\},
$$
where the delays $\tau_k\in [0,T]$ and the coefficients $\mathfrak{c}_k\in\mathbb C$ are
unknown.

Suppose that a set of frequencies $\Psi=\{\psi_1,\ldots,\psi_m\}\subset\Omega$ is assigned to transmitter/receiver pair for channel estimation. The transmitter sends the pilots $P_n,n\in \Psi$ and the    receiver gets vector $\tilde{Y}=(\tilde{Y}_n)_{n\in \Psi}$ with  $\tilde{Y}_n=P_nH_n+\tilde{W}_n$  where $(H_n)$ is determined by the Physics of the communications and depends on the placement of transmitter and receiver in space and there is also some noise  $\tilde{W}=(\tilde{W}_n)$. The  pilot signals are known at the receiver's side and we deal with  their influence  by introducing new quantities $Y_n=P_n^{-1}\tilde{Y}_n,W_n=P_n^{-1}\tilde{W}_n$ and arriving at $Y_n=H_n+W_n.$   When pilots have constant magnitude (say $1$) the noise signal $W_n$   is a phase shifted version of $\tilde{W}_n.$ 
The goal is to recover $H=(H_n)_{n\in\Omega}$
using $Y$. This is a particular case of the problem considered
in~Subsection \ref{subsec_finite_dim}.

The delays may be approximated using the oversampled grid $\{jT/(\rho
N)\}$, $\rho>1$,
so $H$ becomes a sparse vector in the dictionary $\cD(F)$ that consists
of the $\ell_2$-normalized columns of the oversampled Fourier matrix
\begin{equation}
    \label{oversampled}
F_{n,j} = \exp\left(-2\pi i \frac{nj}{\rho N}\right),
\quad 0\le n<N,\;0\le j<\rho N.
\end{equation}
So, we have
$$
H = \sum_{k=1}^K c_k\Phi^{j_k},
\quad \Phi^j := N^{-1/2}F^j \in \cD(F).
$$

Note that due to oversampling we have to deal with a coherent dictionary. E.g. for $\rho=4$ we have $M(\cD(F))\approx 0.9$.
So, the technique of the previous sections is required here.

We define the sub-dictionary $\mathcal D_1$ as the set $\{\Phi^{j_k}:1\le k\le K\}$ of Fourier matrix columns used in the definition of
$H$.

If the delays are random and uniformly distributed on the delay grid then the
inequality
\begin{equation}
    \label{m_condition}
m = |\Psi| \ge C_\rho K^2 \log N
\end{equation}
allows us to estimate $M^+(\cD_1,\cD)$ and guarantees that the condition~\eqref{mu_plus_condition} is fulfilled with high
probability for a random $\Psi$.

Suppose that the noise $W$ satisfies the condition
\begin{equation}
    \label{noise_theta}
    \|W\|_{\cD(F_\Psi)} \le \theta.
\end{equation}
Then Theorem~\ref{thm_mu_plus} guarantees that the OMP will recover all elements
with $|c_k|>\hat\beta := C_1\theta$.

Recall the notions of the NMSE (normalized mean squared error)
and the SNR (signal-to-noise ratio):
$$
\mathrm{NMSE}(H,\hat H):=\frac{\|H-\hat H\|_2^2}{\|H\|_2^2},
\quad \mathrm{SNR}:=\frac{\|H\|_2^2}{\|W\|_2^2}. 
$$
Our considerations show that
\begin{equation}
    \mathrm{NMSE}(H,\hat H)\lesssim
    \mathrm{NMSE}(c,c^{\hat\beta}) + \frac{\|\pi W\|_2^2}{\|W\|_2^2}\mathrm{SNR}^{-1},
\end{equation}
where $c^{\beta}$ is the vector of the coefficients with $|c_k|\le\beta$
zeroed out and $\pi$ is the orthoprojection on the span of dictionary elements of $H$.

The condition~\eqref{noise_theta} is satisfied w.h.p. for a Gaussian noise.
Suppose that $W\sim\mathcal N_{\mathbb C}(0,\sigma^2\mathrm{Id}_{m})$. Then with
high probability we have
$$
\|W\|_{\cD(F_\Psi)}
\le \theta \asymp \sigma\sqrt{\log N}.
$$

Consider a simplified 3GPP   channel model (see \cite{3GPP38901}), when the delays are sampled uniformly
on the oversampled grid $\{jT/(\rho N)\}_{0\le j<\rho N}$ and the powers are decaying exponentially:
$$
    |\mathfrak{c}_k|^2 = \mathrm{const} \cdot \exp(-\alpha\tau_k/T)\cdot 10^{Z_k/10},
    \quad Z_k\sim\mathcal N(0,\zeta).
$$

We have the following corollary.
Suppose that the parameters $\alpha,\zeta,\rho$ and some $p,\delta\in(0,1)$ are fixed; the
numbers $m$, $N$, $K$ are sufficiently large. Then if the
condition~\eqref{m_condition} is satisfied and
$$
\mathrm{SNR}\ge C (K/m)\log N
$$
for sufficiently large $C$, then for a random $\Psi$
and a random channel the OMP algorithm with probability at least $1-p$ returns
an estimate such that
$\mathrm{NMSE}(H,\hat H) \le \delta$.

This technique may be also applied to the so-called  interference case, when
the ``noise'' has two components: the interference (structured)
noise $W_I$  and the thermal (Gaussian) noise $W_T$.

\subsection{A special setting}
\label{ss}

We now discuss the following setting. Let $\cD$ be a dictionary of a Hilbert space $\mathcal{H}$. 
Assume that this dictionary has a coherence parameter $M(\cD)\le \mu <1$. In addition, consider a subdictionary 
$\cD_1 \subset \cD$ with a much smaller coherence parameter $  M(\cD_1)\le \mu_1 <\mu$. 

{\bf Assumptions on $f$.} For given natural numbers $K\le L$, number $q\in (0,1)$, and numbers $0<a<b<\infty$ define 
the class $E(K,L,q,a,b)$ of elements satisfying the conditions
\be\label{FD1}
f\in \Sigma_L(\cD_1), \quad \text{which means}\quad f= \sum_{n=1}^L c_ng_n,\quad g_n \in \cD_1 
\ee
with $\{g_1,\dots,g_L\}$ being distinct elements;
\be\label{FD2}
aq^{n-1} \le |c_n| \le bq^{n-1}, \quad n=1,2,\dots, K;
\ee
\be\label{FD3}
  |c_n| \le bq^{n-1}, \quad n=K+1,\dots, L.
\ee

{\bf Assumptions on parameters.} Assume that our parameters satisfy the following inequalities
\be\label{FD4}
\mu \ge q(1-q) +\mu_1,
\ee
\be\label{FD5}
a(1-q) > b(\mu+q\mu_1),
\ee
\be\label{FD6}
\mu_1 \le (1-q)q^{L-1}.
\ee

\begin{Lemma}\label{FDL1} Assume that inequalities (\ref{FD4}) -- (\ref{FD6}) are satisfied. Then, for an element $f\in E(K,L,q,a,b)$ the PGA chooses the element $g_1$ at the first iteration and 
\be\label{FD7}
f_1/q \in E(K-1,L,q,a,b),\quad f_1 = f -\<f,g_1\> g_1 .
\ee
\end{Lemma}
\begin{proof} First, we bound from below the number $|\<f,g_1\>|$. We have
\be\label{FD8}
|\<f,g_1\>| \ge |c_1| - \sum_{n=2}^L |c_n| \mu_1 \ge a - \frac{bq\mu_1}{1-q}.
\ee
Second, we bound from above the number $|\<f,g\>|$ for any $g\in\cD$ such that $g\neq g_1$. We begin with the case $g\neq g_n$, $n=2,\dots,L$. We have
\be\label{FD9}
|\<f,g\>| \le |c_1||\<g_1,g\>| +\mu \sum_{n=2}^L |c_n|  \le  \frac{b\mu}{1-q}.
\ee
In the case $g=g_j$, $j\in \{2,\dots,L\}$, we have
\be\label{FD10}
|\<f,g_j\>| \le |c_j|  +\mu_1 \sum_{1\le n\le L,n\neq j} |c_n|  \le qb + \frac{b\mu_1}{1-q}.
\ee
By assumption (\ref{FD4}) we obtain 
\be\label{FD11}
qb + \frac{b\mu_1}{1-q} \le \frac{b\mu}{1-q}.
\ee
Inequalities (\ref{FD8}) -- (\ref{FD11}) show that for any $g\in\cD$ such that $g\neq g_1$ we have 
$$
|\<f,g\>| < |\<f,g_1\>|,
$$
which guarantees that  the PGA chooses the element $g_1$ at the first iteration. 

We now check relation (\ref{FD7}). After the first iteration of the PGA we obtain
$$
f_1 = f -\<f,g_1\>g_1 = \sum_{n=2}^L c_ng_n + g_1c_L',\quad c_L' := -\sum_{n=2}^L c_n\<g_n,g_1\>.
$$ 
Denote 
$$
g_n' := g_{n+1}, \quad c_n' := c_{n+1}, \quad n=1,\dots,L-1, \quad g_L' := g_1. 
$$
Then 
$$
f_1 = \sum_{n=1}^L c_n'g_n'
$$
and 
$$
aq^{n-1} \le |c_n'/q| \le bq^{n-1}, \quad n=1,2,\dots, K-1,
$$
$$
  |c_n'/q| \le bq^{n-1}, \quad n=K,\dots, L-1.
$$
It remains to check the $c_L'/q$.  
$$
|c_L'| \le \sum_{n=2}^L |c_n||\<g_n,g_1\>| \le \mu_1\sum_{n=2}^L bq^{n-1} \le \frac{bq\mu_1}{1-q}.
$$
Therefore, we have for $\mu_1$ satisfying (\ref{FD6})
$$
|c_L'/q| \le bq^{L-1}. 
$$
Thus, we have proved that $f_1/q \in E(K-1,L,q,a,b)$.
Lemma \ref{FDL1} is proved. 

\end{proof}

Lemma \ref{FDL1} implies the following Theorem \ref{FDT2}.

\begin{Theorem}\label{FDT2} Assume that inequalities (\ref{FD4}) -- (\ref{FD6}) are satisfied. Then, for an element $f\in E(K,L,q,a,b)$ the PGA chooses the elements $g_1,g_2,\dots,g_K$ at the first $K$ iterations.
\end{Theorem}

Thus, we can apply the following combination of the PGA and the OMP. At the  first $K$ iterations apply the PGA and 
under assumptions of Theorem \ref{FDT2} we recover the elements $g_1,g_2,\dots,g_K$ with large coefficients. 
Then we make an orthogonal projection of $f$ onto $\sp(g_1,g_2,\dots,g_K)$.

\section{Recovery in Banach spaces. Discussion}
\label{Disc}

We begin with presenting some generalization to the case of a Banach space of some concepts introduced in Section \ref{In} in the case of a Hilbert space. For a nonzero element $g\in X$ we let $F_g$ denote a norming (peak) functional for $g$:
$$
\|F_g\|_{X^*} =1,\qquad F_g(g) =\|g\|_X,
$$
where $X^*$ is the dual to $X$ space.
The existence of such a functional is guaranteed by the Hahn-Banach theorem.  The norming functional $F_f$ is a linear functional (in other words is an element of the dual to $X$ space $X^*$), which can be explicitly written in some cases. In a real Hilbert space $F_f$ can be identified with $f\|f\|^{-1}$. In the real $L_p$, $1<p<\infty$, it can be identified with $f|f|^{p-2}\|f\|_p^{1-p}$. Similar expressions are known in the complex spaces as well.  We introduce a new norm, associated with a dictionary $\cD$,  by the formula
$$
\|f\|_\cD:=\sup_{g\in\cD}\sup_{F_g}|F_g(f)|,\quad f\in X.
$$
We note that, in general, a norming functional $F_g$ is not unique. This is why we take $\sup_{F_g}$ over all norming functionals of $g$ in the definition of $\|f\|_\cD$. We do not need $\sup_{F_g}$ in that definition   if for each 
$g\in\cD$ there is a unique norming functional $F_g\in X^*$.   It is known that the uniqueness of the norming functional $F_g$ is equivalent to the property that $g$ is a point of Gateaux smoothness:
 $$
 \lim_{u\to 0}(\|g+uy\|+\|g-uy\|-2\|g\|)/u =0
 $$
 for any $y\in X$. 
 For a Banach space $X$ we define the modulus of smoothness
$$
\rho(u) :=\rho(u,X) := \sup_{\|x\|=\|y\|=1}\left(\frac{1}{2}(\|x+uy\|+\|x-uy\|)-1\right).
$$
The uniformly smooth Banach space is the one with the property
$$
\lim_{u\to 0}\rho(u)/u =0.
$$
In particular, if $X$ is uniformly smooth  then $F_f$ is unique for any $f\neq 0$. 
It turns out that the norm $\|\cdot\|_\cD$ is useful in measuring the noise by imposing the assumption
\be\label{In5}
\|f-f^*\|_\cD \le \epsilon. 
\ee

{\bf Comment on terminology.} We now make a comment on terminology. In the greedy approximation literature we define a dictionary $\cD$ as a system $\{g\}$  of elements $g\in X$ with the following two properties
$$
\|g\|\le 1 \quad\text{for all} \quad g\in \cD \quad \text{and the closure of}\, \sp(\cD) =X.
$$
The normalization condition $\|g\|\le 1$ is imposed for convenience. Clearly, the characteristic $\sigma_m(f,\cD)$ does not depend on normalization. In this paper we mostly use this characteristic. In the case $\|g\|\le B$ for all $g\in \cD$ with some positive constant $B$ we can consider the new system $\cD_B := \{g/B: g\in \cD\}$. Taking into account the fact that the WOMP is homogeneous with respect to $f$ we see that, for instance, Theorem \ref{ssrT1} holds for the dictionary $\cD$ satisfying the condition $\|g\|\le B$. 

Let us discuss the second condition. Suppose that a system $\cS\subset X$ does not satisfy this condition. Then, instead of the Banach space $X$ we consider a subspace $X_\cS$ of $X$, which is the closure (in $X$) of $\sp(\cS)$. This makes the system $\cS$ to be a dictionary in the Banach space $X_\cS$. For this reason, we sometimes with a little abuse of exactness freely use both terms {\it system} and {\it dictionary} for a general system. In the greedy approximation theory there are theorems, which guarantee convergence of certain greedy algorithms with respect to any dictionary $\cD$ for any element $f\in X$. Clearly, in the case, when we deal with a system, we can only apply those theorems to $f\in X_\cS$. 
Also, in many cases, for instance in the case of Theorem  \ref{ssrT1} the proof, which is presented for a dictionary, works for
 any system satisfying the normalization condition.

{\bf General setting. Assumptions on the dictionary. Banach spaces.} In the majority of cases we need to impose restrictions on the dictionary in order to be able to prove theoretical results. We present some standard assumptions on the dictionary. We begin with 
the concept of coherency and give its definition in a general uniformly smooth Banach space. Let $\cD$ be a dictionary in a Banach space $X$. We define the coherence parameter of this dictionary in the following way
$$
M(\cD):= \sup_{g\neq h;g,h\in\cD} |F_g(h)|.
$$
 
\subsection{Fast decaying coefficients} 
\label{s5.1}
  Let $\cD$ be an arbitrary dictionary in a Banach space  $X$. 

\begin{Theorem}\label{FDT1} Assume that for the modulus of smoothness of $X$ we have $\rho(u,X) \le \gamma u^2$. Then for any element $f\in X$, which has an expansion  
$$
f= \sum_{k=0}^\infty a_k g_k,\quad g_k\in \cD,\quad |a_k| \le q^k,\quad k=0,1,\dots,\quad q\in (0,1),
$$
we have
$$
\sigma_m(f,\cD) \le C(\gamma) q^m (1-q)^{-1/2}.
$$
\end{Theorem} 
\begin{proof} We begin with a trivial argument. We have
$$
\sigma_m(f,\cD) \le \left\| \sum_{k=m}^\infty a_k g_k\right\| \le \sum_{k=m}^\infty |a_k| \le q^m (1-q)^{-1}.
$$
We now show how to improve $(1-q)^{-1}$ to the $(1-q)^{-1/2}$. Represent a given $m\in \N$, $m\ge 2$, in the form
$m=m-m_1 +m_1$, where $m_1\in \N$ will be chosen later. We make our approximation in two steps. At the first step we use the $(m-m_1)$-term approximant 
$$
G_{m-m_1}(f) := \sum_{k=0}^{m-m_1-1} a_k g_k\quad\text{and obtain}\quad f':=f-G_{m-m_1}(f) = \sum_{k=m-m_1}^\infty a_k g_k.
$$
Then (see the definition of the $\|\cdot\|_{A_1(\cD)}$ in Subsection \ref{s5.2})
$$
\|f'\|_{A_1(\cD)} \le \sum_{k=m-m_1}^\infty |a_k| \le q^{m-m_1}(1-q)^{-1}.
$$
At the second step we use a known result (see, for instance, \cite{VTbook}, Theorem 6.8, p.342) and obtain
$$
\sigma_m(f,\cD) \le \sigma_{m_1}(f',\cD) \le C(\gamma)(m_1)^{-1/2}\|f'\|_{A_1(\cD)}.
$$
Choosing $m_1$ of the order of $1/\ln(1/q)$ we complete the proof of Theorem~\ref{FDT1}.

\end{proof}

\begin{Remark}\label{FDR1} The upper bound in Theorem \ref{FDT1} is sharp. Indeed, let $\mathcal{H}$ be a Hilbert space with 
an orthonormal basis $\{e_k\}_{k=0}^\infty$ as a dictionary $\cD$. Take 
$$
f= \sum_{k=0}^\infty q^k e_k.
$$
Then, obviously,
$$
\sigma_m(f,\cD) = \left(\sum_{k=m}^\infty q^{2k}\right)^{1/2} = q^m(1-q)^{-1/2}(1+q)^{-1/2}.
$$
\end{Remark}

\subsection{Sparsity assumption A3}
\label{s5.2}
  
 We now present another setting for the study of noisy data. We present some comments in the general setting of greedy approximation in Banach spaces. We only formulate some results for the reader to get a feeling of them and refer the reader for the details to \cite{VT211}, p.48, section Stability.  
Usually, in the greedy algorithms literature the noisy data is understood in the following deterministic way. For a real Banach space $X$ and a dictionary $\cD\subset X$ define $A_1(\cD)$ to be the closure (in $X$) of the convex hull of the symmetrized dictionary $\cD^\pm := \{\pm g : g\in \cD\}$. For each $f\in X$ we associate the following norm
$$
\|f\|_{A_1(\cD)} := \inf \{M>0:\, f/M\in A_1(\cD)\}.
$$

Take a number $\e\ge 0$ and two elements $f$, $f^*$ from a Banach space $X$ such that
\be\label{S1}
\|f-f^*\| \le \e,\quad
f^*/B \in A_1(\cD),
\ee
with some number $B>0$.
Then we interpret $f$ as a noisy version of our original signal $f^*$, for which we know that it has some good properties formulated in terms of $A_1(\cD)$. The first results on approximation of noisy data (in the sense of (\ref{S1})) were obtained in \cite{VT115} for the Weak Chebyshev Greedy Algorithm (WCGA) and the Weak Greedy Algorithm with Free Relaxation (WGAFR). The WCGA is the generalization to the case of Banach spaces of the Weak Orthogonal Matching Pursuit (WOMP) defined in Hilbert spaces (see above). 
Later, in \cite{VT165} we proved Theorem \ref{ST1}, which 
covers all algorithms from the collection Weak Biorthogonal Greedy Algorithms (WBGA($\tau$)). Thus, Theorem \ref{ST1} covers the known results from  \cite{VT115} on WCGA and WGAFR and also covers the corresponding results on some other greedy algorithms.

\begin{Theorem}[{\cite{VT165}}]\label{ST1} Let $X$ be a uniformly smooth Banach space with modulus of smoothness $\rho(u)\le \gamma u^q$, $1<q\le 2$. Assume that $f$ and $f^*$ satisfy (\ref{S1}).
Then, for any algorithm from the collection  WBGA($\tau$ with the weakness sequence $\tau =\{t_k\}_{k=1}^\infty$, applied to $f$ we have
$$
\|f_m\| \le  \max\left\{2\e,\, C(q,\gamma)(B+\e) \Big(1+\sum_{k=1}^mt_k^p\Big)^{-1/p}\right\},
\quad p:=\frac{q}{q-1},
$$
with $C(q,\gamma)= 4(2\gamma)^{1/q}$.
\end{Theorem}

The corresponding version of Theorem \ref{ST1} (for the OMP) for Hilbert spaces reads as follows.

\begin{Theorem}\label{DiT1} Let $\mathcal{H}$ be a  Hilbert space.   Assume that $f$ and $f^*$ satisfy (\ref{S1}).
Then, for the OMP applied to $f$ we have
$$
\|f_m\| \le  \max\left\{2\e,\,  4(B+\e) (1+ m)^{-1/2}\right\}.
$$
\end{Theorem}

 \subsection{Performance of the WQOGA}
 \label{QO}

 It is well known that the Lebesgue-type inequalities guarantee exact recovery of sparse elements. 
In this section we discuss greedy type algorithms, which are based on other ideas than dual greedy and $X$-greedy algorithms are based on, but provide good results for exact recovery and Lebesgue-type inequalities for special dictionaries. Results of this section are based on the paper \cite{ST}, which is a follow up to the dissertation \cite{S}. 
We study here the following greedy-type algorithm.

 {\bf Weak Quasi-Orthogonal Greedy Algorithm (WQOGA).} 
 Let $t\in (0,1]$. Denote $f_0:=f_0^{q,t}:=f$ (here and below index $q$ stands for {\it quasi-orthogonal}) and find $\varphi_1:=\varphi_1^{q,t}\in\cD$ such that
 $$
 |F_{\varphi_1}(f_0)| \ge t\sup_{g\in \cD}|F_g(f_0)|.
 $$
 Next, we find $c_1$ satisfying
 $$
 F_{\varphi_1}(f-c_1\varphi_1)=0.
 $$ 
 Denote $f_1:=f_1^{q,t}:=f-c_1\varphi_1$. 
 
 We continue this construction in an inductive way. Assume that we have already constructed residuals $f_0,f_1,\dots,f_{m-1}$ and dictionary elements $\varphi_1,\dots,\varphi_{m-1}$. Now, we pick an element
 $\varphi_m:=\varphi_m^{q,t}\in\cD$ such that   
 $$
 |F_{\varphi_m}(f_{m-1})| \ge t\sup_{g\in \cD}|F_g(f_{m-1})|.
 $$
 Next, we look for $c_1^m,\dots,c_m^m$ satisfying
 \begin{equation}\label{QO9.1}
 F_{\varphi_j}(f-\sum_{i=1}^mc_i^m\varphi_i)=0,\quad j=1,\dots,m. 
 \end{equation}
If there is no solution to (\ref{QO9.1}) then we stop, otherwise we denote $G_m:=G_m^{q,t} := \sum_{i=1}^mc_i^m\varphi_i$ and $f_m:=f_m^{q,t}:=f-G_m $ with $c_1^m,\dots,c_m^m$ satisfying (\ref{QO9.1}).

\begin{Remark} We note that (\ref{QO9.1}) has a unique solution if 
$$
\det [F_{\varphi_j}(\varphi_i)]_{i,j=1}^m \neq 0.
$$
 We apply the WQOGA in the case of a dictionary with the coherence parameter $M:=M(\cD)$. Then, by a simple well known argument on the linear independence of the rows of the matrix $[F_{\varphi_j}(\varphi_i)]_{i,j=1}^m$, we conclude that (\ref{QO9.1}) has a unique solution for any $m<1+1/M$.   Thus, in the case of an $M$-coherent 
dictionary $\cD$, we can run the WQOGA for at least $[1/M]$ iterations.
\end{Remark} 

In the case $t=1$ we call the WQOGA the Quasi-Orthogonal Greedy Algorithm (QOGA). In the case of QOGA we need to make an extra assumption that the corresponding maximizer $\ff_m\in\cD$ exists. Clearly, it is the case when $\cD$ is finite.  

It was proved in \cite{VT107}  (see also \cite{VTbook}, p.382) that the WQOGA is as good as the WOGA in the sense of exact recovery of sparse signals with respect to incoherent dictionaries. The following result was obtained in \cite{VT107} (see Theorem 11.14 there).
 \begin{Theorem}[{\cite{VT107}}]\label{QOT1} Let $t\in (0,1]$. Assume that $\cD$ has coherence parameter $M$. Let $S<\frac{t}{1+t}(1+1/M)$. Then for any
 $f$ of the form
 \begin{equation}\label{QO1.1}
 f=\sum_{i=1}^Sa_i\psi_i,
 \end{equation}
 where $\psi_i$ are distinct elements of $\cD$, the WQOGA recovers it exactly after $S$ iterations. In other words we have that $f^{q,t}_S=0$.
 \end{Theorem}
It is known (see, for instance, \cite{VTbook}, pp.303--305) that the bound $S<\frac{1}{2}(1+1/M)$ is sharp for exact recovery by the OGA. 

 As above,  we define the best $m$-term approximation in the norm $Y$ as follows
$$
\sigma_m(f)_Y := \inf_{g\in\Sigma_m(\cD)}\|f-g\|_Y.
$$
In this section the norm $Y$ will be either the norm $X$ of our Banach space or the norm $\|\cdot\|_\cD$ defined above.
In \cite{ST} (see Theorems 1.4, 1.5 and Corollary 1.2 there) we proved the following two Lebesgue-type inequalities.
 \begin{Theorem}[{\cite{ST}}]\label{QOT0.1} Assume that $\cD$ is an $M$-coherent dictionary. Then for $m\le 1/(3M)$ we have for the QOGA
 \begin{equation}\label{QO0.2}
 \|f_m\|_\cD \le 13.5\sigma_m(f)_\cD.
 \end{equation}
 \end{Theorem}

 \begin{Theorem}[{\cite{ST}}]\label{QOT0.2} Assume that $\cD$ is an $M$-coherent dictionary in a Banach space $X$. There exists an absolute constant $C$   such that for $m\le 1/(3M)$ we have for the QOGA
$$
\|f_m\|_X \le C \inf_{g\in\Sigma_m(\cD)}(\|f-g\|_X + m\|f-g\|_\cD).
$$
\end{Theorem} 
\begin{Corollary}[{\cite{ST}}]\label{QOC0.2} Using the inequality $\|g\|_\cD \le \|g\|_X$ we obtain from Theorem \ref{QOT0.2}
$$
\|f_m\|_X \le C(1+m) \sigma_m(f)_X.
$$
\end{Corollary}

\begin{Theorem}[{\cite{ST}}]\label{QOT1.5} Assume that $\cD$ is an $M$-coherent dictionary in a Hilbert space $\mathcal{H}$. Then for $m\le \frac{2}{3}\frac{t}{1+t}\frac{1}{M}$ we have the WQOGA
$$
\|f_m\|_\mathcal{H} \le C_6(t)\inf_{g\in\Sigma_m(\cD)}(\|f-g\|_\mathcal{H} + m^{1/2}\|f-g\|_\cD).
$$
\end{Theorem} 

\begin{Corollary}[{\cite{ST}}]\label{QOC2} Using the inequality $\|g\|_\cD \le \|g\|_X$ we obtain from Theorem \ref{QOT1.5} that under condition $m\le \frac{2}{3}\frac{t}{1+t}\frac{1}{M}$ we have for the WOGA in a Hilbert space $\mathcal{H}$
$$
\|f_m\|_\mathcal{H} \le C(1+m^{1/2}) \sigma_m(f)_\mathcal{H}.
$$
\end{Corollary}

Inequality (\ref{QO0.2}) is a perfect (up to a constant 13.5) Lebesgue-type inequality. It indicates that the norm $\|\cdot\|_\cD$ used in that inequality is a suitable norm for analyzing performance of the QOGA. Corollary \ref{QOC0.2} shows that the Lebesgue-type inequality (\ref{QO0.2}) in the norm $\|\cdot\|_\cD$ implies the Lebesgue-type inequality in the norm $\|\cdot\|_X$. 
We do not know if the Lebesgue-type inequality in Corollary  \ref{QOC0.2}
 is sharp. We know that Corollary  \ref{QOC2} is sharp.  It is known (see \cite{GMS}) that the factor $(1 + m^{1/2})$  cannot
be replaced by a slower growing in $m$ factor.
  Thus, in the case of Hilbert spaces, the technique from \cite{ST} based on the general
inequality (\ref{QO0.2}) provides sharp Lebesque-type inequalities.

{\bf Comment \ref{QO}.1.} In the paper \cite{ST} we gave an argument in favour of consideration of greedy approximation in Banach spaces. We introduced a concept of $M$-coherent dictionary in a Banach space which is a generalization of the corresponding concept in a Hilbert space.  We analysed the Quasi-Orthogonal Greedy Algorithm (QOGA), which is a generalization of the Orthogonal Greedy Algorithm (Orthogonal Matching Pursuit) for Banach spaces. It is known (see \cite{VT107}) that the QOGA recovers exactly $S$-sparse signals after $S$ iterations provided $S<(1+1/M)/2$. This result is well known for the Orthogonal Greedy Algorithm in Hilbert spaces. The following question is of great importance: Are there dictionaries in $\mathbb R^n$ such that their coherence in $\ell_p^n$ is less than their coherence in  $\ell_2^n$ for some $p\in (1,\infty)$? In Section 3 of \cite{ST} we showed that the answer to the above question is "yes". Thus, for such dictionaries, replacing the Hilbert space $\ell_2^n$ by a Banach space $\ell_p^n$, we improve an upper bound for sparsity that guarantees an exact recovery of a signal. We note that the computational complexity of the QOGA in the case 
  of $\ell_p^n, 1<p<\infty$, is close to that of the OGA in the case of $\ell_2^n$ because we can write the formula for the norming functional explicitly.


\begin{thebibliography}{xxxx}

\bibitem{App11}
Applebaum, Lorne, WaheedU. Bajwa, A. Robert Calderbank, et al. “Deterministic pilot sequences for sparse channel estimation in OFDM systems.” IEEE 2011 17th International Conference on Digital Signal Processing (DSP), 2011

\bibitem{Bour} J. Bourgain, An improved estimate in the restricted isometry problem, In Geometric Aspects of Functional Analysis, volume 2116 of Lecture Notes in Mathematics, pages 65--70. Springer, 2014.

 \bibitem{BDJR} S. Brugiapaglia, S. Dirksen, H.C. Jung, and H. Rauhut,  Sparse recovery in bounded Riesz systems with applications to numerical methods for PDEs, Applied and Computational Harmonic Analysis, {\bf 53}  (2021), 231-269.







\bibitem{DPTT} F. Dai, A. Prymak, V.N. Temlyakov, and S.U. Tikhonov, Integral norm discretization and related problems, {\it Uspekhi Mat. Nauk}, {\bf 74} (2019), no. 4(448), 3--58; translation in {\it Russian Math. Surveys}, {\bf 74} (2019), no. 4, 579--630.

\bibitem{DKT} F. Dai, E. Kosov, and V. Temlyakov, A survey of sampling discretization and related topics, submitted (2023).

 \bibitem{DTM2} F. Dai and V.N. Temlyakov, Random points are good for universal discretization, J. Math. Anal. Appl., {\bf 529} (2024) 127570; arXiv.2301.12536[math.FA] 5 Feb 2023.


\bibitem{VT165} A.V. Dereventsov and V.N. Temlyakov, A unified way of analyzing some greedy algorithms, Journal of Functional Analysis, {\bf 277}(12) (2019), 108286.



\bibitem{DET}  D. Donoho, M. Elad and V.N. Temlyakov,  On the Lebesgue type inequalities for greedy approximation,  J. Approximation Theory, {\bf 147} (2007), 185--195.



\bibitem{FR} S. Foucart and H. Rauhut, A Mathematical Introduction to Compressive Sensing,
  Birkh{\"a}user, 2013. 


 
\bibitem{GMS}  A.C. Gilbert, S. Muthukrishnan and M.J. Strauss, Approximation of functions over redundant dictionaries using coherence, {\em The 14th Annual ACM-SIAM Symposium on Discrete Algorithms},  2003. 

   \bibitem{HR} I. Haviv and O. Regev, The restricted isometry property of subsampled Fourier matrices, In Geometric aspects of functional analysis, volume 2169 of Lecture Notes in Math., pages 163?179. Springer, Cham, 2017.
   
   
   \bibitem[HL18]{HL18} 
Heath Jr R. W., Lozano A. Foundations of MIMO communication. Cambridge University Press, 2018.






\bibitem{KKLT} B.S. Kashin, E. Kosov, I. Limonova, and V.N. Temlyakov, Sampling discretization and related problems, {\it J. Complexity}, {\bf 71} (2022), 101653.


\bibitem{LMT} I. Limonova, Yu. Malykhin, and V. Temlyakov, One-sided discretization inequalities and sampling recovery, {\it Uspekhi Mat. Nauk}, {\bf 79} (2024), no. 3(477), 149--180.


\bibitem{LivTem} E.D. Livshitz and V.N. Temlyakov, Sparse approximation and recovery by greedy algorithms, 
IEEE Transactions on Information Theory, {\bf 60} (2014), 3989--4000; arXiv:1303.3595v1 [math.NA] 14 Mar 2013.


 \bibitem{S} D. Savu, Sparse Approximation in Banach Spaces, Ph. D. Dissertation, University of South Carolina (2009).

\bibitem{ST} D. Savu and V. Temlyakov, Lebesgue-type inequalities for greedy approximation in Banach spaces, IEEE Transactions on Information Theory, {\bf 59} (2013), 1098--1106. 



\bibitem{T1} V.N. Temlyakov,  \emph{Weak greedy algorithms},
Adv. Comput. Math. \textbf{12}  (2000), 213--227.



\bibitem {VT107}  V.N. Temlyakov, Greedy Approximations,   Foundations of Computational Mathematics, Santander 2005,  London Mathematical Society Lecture Notes Series, {\bf 331}, Cambridge University Press, 371--394 (2006). 

\bibitem{VT115}  V.N. Temlyakov, Relaxation in greedy approximation, {\em  Constructive Approximation}, {\bf 28}   (2008), 1--25.

  \bibitem{VTbook} V.N. Temlyakov, Greedy Approximation, Cambridge University
Press, 2011.



\bibitem{VTbookMA} V.N. Temlyakov,   Multivariate Approximation, Cambridge University Press, 2018.


\bibitem{VT211} V.N. Temlyakov, Brief introduction in greedy approximation, Uspekhi Mat. Nauk, {\bf 80} (2025), 23--104; arXiv:2502.13432v1 [math.NA] 19 Feb 2025.



\bibitem{Tr06}
    J.~Tropp,
    ``Just Relax: Convex Programming Methods for Identifying Sparse Signals in Noise'',
    \textit{IEEE Trans. Inf.Th.}, \textbf{52}:3 (2006).

\bibitem{TGS06}
    J.~Tropp, A.~Gilbert, and M.~Strauss,
    ``Algorithms for simultaneous sparse approximation part I: greedy pursuit'',
    \textit{Signal processing}, \textbf{86}:3 (2006), 589--602.



\bibitem{XRS}
He, Xueyun, Rong fang Song, and Wei-Ping Zhu. \emph{Pilot allocation for distributed-compressed-sensing-based sparse channel estimation in MIMO-OFDM systems}, IEEE Transactions on Vehicular Technology 65.5 (2015): 2990-3004.


\bibitem{Z} T. Zhang, \emph{Sparse Recovery with Orthogonal Matching Pursuit under RIP}, IEEE Transactions on Information Theory, {\bf 57} (2011), 6215--6221.


\bibitem{Zh16} Yi Zhang, et al. ``Novel Compressed Sensing-Based Channel Estimation Algorithm and Near-Optimal Pilot Placement Scheme'', IEEE Transactions on Wireless Communications, Vol 15, No.4, 2016



\bibitem[3GPP 38.901]{3GPP38901}
    3rd Generation Partnership Project,
    Technical Specification Group Radio Access Network;
    Study on channel model for frequencies from 0.5 to 100 GHz
    (Release 19).


\end{thebibliography}
\end{document}